\documentclass[a4paper,12pt]{article}
\usepackage[utf8x]{inputenc}
 
 \usepackage{tabularx} 
\usepackage{calc}

\usepackage{amsmath, latexsym, amsfonts, amssymb, amsthm, amscd,graphicx,color,epsfig,eurosym}
\usepackage[english]{babel}
\usepackage{graphics,epsf,psfrag,graphicx,epsfig}

   \newtheorem{theorem}{Theorem}[section]
   \newtheorem{assum}[theorem]{Assumption}
    \newtheorem{lem}[theorem]{Lemma}
   \newtheorem{prop}[theorem]{Proposition}
   
     \newtheorem{defi}[theorem]{Definition}

 \newcommand{\ga}{\alpha}
\newcommand{\gb}{\beta}
\newcommand{\gd}{\delta}
\newcommand{\gep}{\varepsilon}       

\newcommand{\gk}{\kappa}

\newcommand{\gl}{\lambda}

\newcommand{\gs}{\sigma}

\newcommand{\bP}{{\ensuremath{\mathbf P}} }
\newcommand{\bE}{{\ensuremath{\mathbf E}} }

\newcommand{\cA}{{\ensuremath{\mathcal A}} }
\newcommand{\cF}{{\ensuremath{\mathcal F}} }

\newcommand{\cE}{{\ensuremath{\mathcal E}} }

\newcommand{\cC}{{\ensuremath{\mathcal C}} }

\newcommand{\cD}{{\ensuremath{\mathcal D}} }
\newcommand{\cB}{{\ensuremath{\mathcal B}} }
\newcommand{\cI}{{\ensuremath{\mathcal I}} }

\newcommand{\tS}{{\ensuremath{\tilde S}} }
\newcommand{\ttau}{{\ensuremath{\tilde\tau}} }

\newcommand{\bbR}{\mathbb{R}}
\newcommand{\Z}{\mathbb{Z}}
\newcommand{\N}{\mathbb{N}}
\newcommand{\ind}{\mathbf{1}}

\title{ The scaling limits of a heavy tailed renewal Markov process. }
\author{Sohier Julien \thanks{Universit\'{e} Paris-Dauphine, Ceremade, CNRS UMR 7534, 
   F 75016 Paris France.  e-mail: jusohier@gmail.com }}

\begin{document}

\maketitle

 \begin{abstract}

 In this paper we consider heavy tailed Markov renewal processes and we prove that, suitably renormalised, 
 they converge in law towards the $\ga$-stable regenerative set. 
We then apply these results to the strip wetting model which is a random walk $S$ constrained above a wall and rewarded
or penalized when it hits the strip $[0,\infty) \times [0,a]$ where $a$ is a given positive number.
  The convergence result that we establish allows to characterize the scaling limit of
this process at criticality. 
    

 \end{abstract}

  

    \section{Introduction and main results.}\label{sec1}

   \subsection{The Markov renewal setup.}  
 
   This work mainly focuses on the scaling limits of a Markov renewal process with heavy tailed renewal times.
 By a Markov renewal process, we mean a point process $\tau$ on $\N$ whose increments $\tau_1, \ldots, \tau_n-\tau_{n-1}, \ldots$ 
 are not necessarily
 i.i.d., but governed by a Markov chain $J$ with measurable state space $(\cE,\mu)$. Throughout this paper, 
 we will assume that $(\cE,\mu)$ is a compact Polish space equipped with its standard Borel $\gs$-field $\cB(\cE)$. 
 
  The Markov dependence of $J$ on $\tau$ may be described in 
 various equivalent ways, one of them is the following:
  \begin{enumerate}
   \item First consider a kernel of the form $K_{x,dy}(n)$ with $x,y \in \cE, n \in \Z^+$, that is a function 
  $K: \cE\times \cB(\cE) \times \Z^{+} \mapsto [0,1]$ 
such that for fixed $ n \in \Z^{+}$, $K_{\cdot,\cdot}(n)$ verifies the following: 
   \begin{itemize}
   \item for each $x \in \cE$, $K_{x,\cdot}(n)$ is a $\gs$-finite measure on $\cE$.
    \item for every $F \in \cB(\cE)$, $K_{\cdot,F}(n)$ is a Borel function.
  \end{itemize}
   We will always assume that for every $n \in \N$ and $x \in \cE$, $K_{x,dy}(n)$ has a density
 with respect to $\mu(\cdot)$ that we denote by
 $k_{x,y}(n)$, implying in particular that 
  \begin{equation}
   k(x,y) := \frac{1}{\mu(dy)} \sum_{n \geq 1} K_{x,dy}(n) = \sum_{n \geq 1} k_{x,y}(n)
  \end{equation} 
 is well defined. 
  \item Then sample $J$ as a Markov chain on $\cE$ starting from some initial point $J_0 = x_0 \in \cE$ and with transition kernel
    \begin{equation}
     \bP_{x_0}[ J_{n+1} \in dy|J_n = x] := k(x,y) \mu(dy).
    \end{equation} 
   \item Finally sample the increments $(\tau_i-\tau_{i-1})_{i \geq 1}$ as a sequence of independent, but not identically distributed 
 random variables according to the conditional law:
    \begin{equation}
     \bP_{x_0}[ \tau_{n+1} - \tau_{n} = m|\{J_i\}_{i \geq 0} ] = \frac{k_{J_n,J_{n+1}}(m)}{ k(J_n,J_{n+1})}
 , { \hspace{.2 cm}  } n,m \geq 1.
    \end{equation} 
 
  \end{enumerate}

  Markov renewal processes have
  been introduced  independently by L\'{e}vy \cite{Le} and Smith \cite{Sm}, and their basic properties have been studied  
 by Pyke \cite{Py}, Cinlar \cite{Cin2} and Pyke and Schauffele \cite{PyS} among others. A modern reference  is 
 Asmussen \cite{As}[VII,4], which describes some applications related to these processes, in particular in
 the field of queuing theory. 
 More recently,  they have been widely used
  as a convenient tool to describe phenomena arising in models 
 issued from statistical mechanics, and more particularly in models related to pinning models, such as 
 the periodic wetting model  (\cite{CGZ1} and the monograph \cite{GB}[Chapter 3]) or the $1+1$ dimensional 
 field with Laplacian interaction (\cite{cd1} and \cite{cd2}).

 We will show our results in the case where the kernel $K$ satisfies the following assumptions.
  \begin{assum}\label{as1} We make the following assumptions on the kernel $K$: 
\begin{enumerate}
 \item[$\bullet$] the transition function
 of the Markov chain $J$ is absolutely continuous with respect to $\mu$,
 and its density $k(\cdot,\cdot)$ is continuous on 
 the whole square $\cE^{2}$.

\item[$\bullet$]
  There exist $\ga \in (0,1)$,  a strictly positive continuous bounded function $\Phi(\cdot)$ on $\cE^{2}$ and 
 $L(\cdot)$  a slowly varying function such that  the equivalence
 \begin{equation}\label{BASEQ}
  k_{x,y}(n) \sim \Phi(x,y) \frac{L(n) } {n^{1 + \ga}}
 \end{equation} 
   holds uniformly for $(x,y) \in \cE^{2}$; moreover, we assume that the Markovian renewal is non defective, that is 
for every $x \in \cE$: 
  \begin{equation}
   \int_{ y \in \cE} \sum_{n \geq 1} k_{x,y}(n) \mu(dy) = 1.
    \end{equation} 
 
 \end{enumerate}
  \end{assum}

 A few remarks about these assumptions are in order: 
 
 \begin{enumerate}
  \item[$\bullet$] We will use the following notation which is valid for every measurable $A \in \gs(\tau)$:
  \begin{equation}
   \frac{1}{\mu(dy)} \bP[ A, J_1 \in dy] =: \bP[ A, J_1 = y].
  \end{equation} 
    The fact that $k$ is continuous implies that this definition makes sense everywhere,
 and not only almost everywhere.
 
  \item[$\bullet$]  The strict positivity of the function $\Phi$ implies the strict positivity of the transition 
 kernel of $J$ on the whole $\cE^{2}$. In particular, $J$ is a 
 \textit{regular} Markov chain (see \cite{Fel2}[VIII,7]). Regular Markov chains with strictly positive density are arguably the simplest example of Markov chains with 
 continuous state space which still satisfy the basic ergodic theorem. More precisely, it is well known that 
 every strictly positive regular kernel on a closed interval is ergodic, and also that the ergodicity of such a kernel 
 is equivalent to the existence of a strictly positive stationary distribution which we will denote by $\Pi$ (see \cite{Fel2}[VIII,7, Theorems 1 and 2]).
 \end{enumerate}

  A consequence of equation \eqref{BASEQ} is the fact that the Markov renewal process $\tau$ has steps distribution
 with \textit{infinite mean}. Making use of the Markov renewal theorem (see \cite{GrW} for this result in the case where 
 $J$ is Harris recurrent, which is the case in our regular setup),
  this implies in particular that, as $n \to \infty$,
 the Green function associated to $\tau$ verifies:
 \begin{equation}
  \bP[ n \in \tau] \to 0. 
 \end{equation}

 We point out that even for 
 ordinary renewal processes, the exact rate of decay of the Green function in the case of infinite mean has been a
  longstanding problem which has been solved in great generality only recently by Doney (see \cite{Do}). A consequence of Doney's results 
 is the fact that, given a true renewal $\tau$ with interarrival times given by 
  \begin{equation}
   \bP[\tau_1 = n] = \frac{L(n)}{n^{1+\ga}} 
  \end{equation}  
  where $L(\cdot)$ is some slowly varying function and $\ga \in (0,1)$, the following equivalence holds:
 \begin{equation}\label{DoMM}
   \bP[ n \in \tau] \sim \frac{\ga \sin(\ga \pi)}{\pi} \frac{1}{L(n)n^{1-\ga}}.
 \end{equation}  
  
 The extension of this deep 
result to the framework of Markov renewal theory has been achieved in the case of finite state modulating chain in \cite{CGZ}, but 
 it turns out that their techniques do not extend to the continuous state case we are considering here. 

  As a matter of fact, our first 
  result deals with the integrated behavior of the Markov renewal function. More specifically, for $(x,y,n) \in \cE^{2} \times \N$,
 let us
 denote by $U(n,x,dy)$ the  Markov renewal mass distribution function defined
  by $U(n,x,dy) := \sum_{i=0}^{n} \sum_{k \geq 1}\bP_x[\tau_k =i, J_k \in dy] =: \sum_{i=0}^{n} u(i,x,dy)$.  
 The purpose of the section \ref{sec2} of the present work is to give asymptotics 
 on $U$ for large $n$, generalizing to our setup the well known Tauberian theorems which hold
 for the classical renewal. We show the following:
  \begin{theorem}\label{mP2}
    As $n \to \infty$, for every $x \in \cE$, the following weak convergence holds :
  \begin{equation}
  \frac{U(n,x,dy)L(n)}{n^{\ga}} \to \frac{\ga}{\Gamma(1+\ga) \Gamma(1-\ga)}\frac{\Pi[dy]  }{  \bE_{\Pi^{(2)}}[\Phi/k]  }
  \end{equation} 
  where $\Pi^{(2)}$ is the invariant measure associated to the Markov chain $(J_i,J_{i+1})_{i \geq 0}$.
  \end{theorem}

  By this we mean that for every continuous bounded function $f$ on $\cE$ and every fixed
 $x \in \cE$, as $n \to \infty$, the following convergence holds :
   
 \begin{equation}
  \frac{L(n)}{n^{\ga}} \int_{y \in \cE}  U(n,x,dy) f(y) \to
  \frac{\ga}{\Gamma(1+\ga) \Gamma(1-\ga)}\frac{\int_{y \in \cE} f(y) \Pi[dy]  }{  \bE_{\Pi^{(2)}}[\Phi/k]  }
 \end{equation} 
  
 Note  that $d\Pi^{(2)}[u,v] = \Pi(du) k(u,v) \mu(dv)$ where $\Pi(\cdot)$ is the invariant measure 
 associated to the Markov chain $J$, and thus in particular one has the equality 
  
 \begin{equation}\label{Inv}
  \bE_{\Pi^{(2)}}[\Phi/k]  = \int_{\cE^{2}} \Phi(u,v) \Pi(du) \mu(dv).
 \end{equation}

 It turns out that Theorem \ref{mP2} is sufficient to resolve the large scale behavior of
 the set $\tau$, which we describe by considering the class $\cC_{\infty}$ of closed subsets of $[0,\infty)$ endowed 
 with the Mathéron topology. This topology can be defined as follows: for $t \geq 0, F \in \cC_{\infty}$, we set
 $d_t(F) := \inf (F \cap (t,\infty))$. The function
 $t \mapsto d_{t}(F)$ is right continuous and $F$ may be recovered from $d_{\cdot}(F)$ as
 $F = \{ t \in \bbR^{+}, d_{t^{-}}(F) = 0 \}$. The space of c\`{a}dl\`{a}g functions may be endowed with the standard Skorohod 
 topology, and this topology gives the Mathéron topology \textit{via} the previous identification. Under this topology, $\cC_{\infty}$
 is metrizable, separable and compact,  in particular it is a Polish space. 

 Finally, we introduce the classical notion of $\ga$ stable regenerative set; recall that a subordinator is a non decreasing 
 Lévy process. It is said to be $\ga$-stable ($\ga \in (0,1)$) if its Lévy-Khintchine exponent $\phi(\gl)$ is equal to 
 $\gl^{\ga}$. We consider $\cA_{\ga}$, the  $\ga$ stable regenerative set,
  which is defined as being the closure of the range of 
 the $\ga$ stable subordinator. We stress that $\cA_{\ga}$ is a highly non trivial random element 
 from $\cC_{\infty}$, 
which coincides for $\ga = 1/2$ with the zero level set of a standard brownian motion. 
  
   If we consider $\tau$ as a subset of $\N$, the set $\tau_{(N)} := \frac{\tau \cap [0,N]}{N}$ may be viewed as 
a (random) element of $\cC_{\infty}$. If the Markov renewal process $(\tau,J)$ satisfies Assumption \ref{as1}, we show that the following result holds, it 
 is actually the first main result of this paper. 
 
 \begin{theorem}\label{main1}
   The sequence of rescaled sets $(\tau_{(N)})_{N}$ converges in law towards the set $\cA_{\ga} \cap [0,1] $.
 \end{theorem}

  A proof of such a result in the classical renewal framework can be found in \cite{GB}[Theorem 2.7] by making use of the 
 pointwise 
 convergence of the L\'{e}vy exponent of a Poisson process whose range is equal to $\tau_{(N)}$ towards the one of the 
 $\ga$ stable-regenerative set, which directly implies the result (see \cite{Fi});  the same idea is not available in our setup 
 since the increments of $\tau$ are no longer i.i.d. Another idea which does
 not work here has been used in \cite{CGZ1} to prove a very similar result in the  context
 of weakly inhomogeneous wetting models; the authors use in an essential way the finiteness of the state space of the governing 
 Markov chain $J$, and the local behavior of the Green function given in equation \eqref{DoMM}.

 \subsection{Application to the strip wetting model.}

   \subsubsection{ Definition of the model. } 

   The main motivation for proving Theorem \ref{mP2} has been provided by a model which originates from statistical mechanics and 
 which we describe now. 
 
 We consider $ (S_n)_{n \geq 0}$ a random walk such that $S_0 := 0$ and $S_n := \sum_{i = 1}^n X_i$ where 
the $X_i$'s are i.i.d. and $X_1$ has a density $h(\cdot)$ with respect to the Lebesgue measure. 
  We denote by $\bP$ the law of $ S$, and by $\bP_x$ the law of the same process starting from $x$.
  We will assume that $h(\cdot)$ is continuous and bounded on $\bbR$, that $h(\cdot)$ is positive in a neighborhood of the origin,
  that $\bE[X] = 0$ and that $\bE [X^2]=: \gs^2 \in (0,\infty)$.  
 
    The fact that $h$ is continuous and positive in the neighborhood of the origin entails that
 \begin{equation}\label{hypn}
  n_0 := \inf_{n \in \mathbb{N}} \{ (\bP[S_{n} > a],\bP[-S_{n} > a]) \in (0,1)^{2} \} < \infty.
 \end{equation} 
  In the sequel, we will assume that $n_0 = 1$  (and thus 
 that $ (\bP[-S_1 > a],\bP[S_1 > a]) \in (0,1)^{2}$). 
 We stress that our work could be easily extended to the generic $n_0 \geq 2$ case, althought this should lead to 
 some specific technical difficulties (for example one should extend Theorem \ref{main1} to the case where the transition 
 function of $J$ may vanish on $\cE$).

  For $N$ a positive integer, we consider the event $\cC_N := \{ S_ 1 \geq 0, \ldots, S_N \geq 0 \}$.
We define the probability law $\bP_{N,a,\gb}$ on $\bbR^N$ by 
 \begin{equation}
  \frac{d\bP_{N,a,\gb}}{d\bP} := \frac{1}{Z_{N,a,\gb}} \exp\left( \gb \sum_{k=1}^N \ind_{ S_k \in [0,a] } \right) \ind_{ \cC_N} 
\end{equation}
 where  $\gb \in \bbR$ and $Z_{N,a,\gb}$ is the normalisation constant usually called the partition function of the system.
 For technical reasons, for $x \in [0,a]$, we will sometimes need to consider the quantity 
 \begin{equation}
  Z_{N,a,\gb}(x) := \bE_{x} \left[ \exp\left( \gb \sum_{k=1}^N \ind_{ S_k \in [0,a] } \right) \ind_{ \cC_N}  \right].
 \end{equation} 
  Note that $Z_{N,a,\gb} = Z_{N,a,\gb}(0)$.
   
  $\bP_{N,a,\gb}$ is a  $(1+1)-$dimensional model for a linear chain of length $N$ which is 
 attracted or repelled to a defect \textit{strip} of width $a$. By $(1+1)-$dimensional, we mean that the configurations of the linear
chain are described by the trajectories $(i,S_{i})_{i \leq N}$ of the walk, so that we are dealing
with directed models. The strength of this interaction with the strip is tuned by 
 the parameter $\gb$. Regarding the terminology, note that the use of the term \textit{wetting} has become customary to describe 
 the positivity constraint $\cC_N$ and refers to the interpretation of the field as an effective model for the interface of separation 
 between a liquid above a wall and a gas, see \cite{DGZ}. 

 The purpose of this part is to investigate the behavior of $ \bP_{N,a,\gb}$ in the large $N$ limit.
 In particular, we would like to understand when the reward $\gb$ is strong enough to pin the chain near the defect strip, a 
 phenomenon that we will call \textit{localization}, and what are the macroscopic effects of the reward on the system.
  We point out that this kind of questions have been answered 
 in depth in the case of the
 standard wetting model, that is formally in the $a = 0$ case, and that the problem of extending these results to 
 the case where the defects are in a strip 
 has been posed in  \cite[Chapter 2]{GB}.  Note that a closely related pinning model in continuous time has been considered 
and resolved 
 in \cite{CrKMV}; we stress however that their techniques are very peculiar to the continuous time setup.

%

   \subsubsection{The free energy.} A standard way to define localization for our model is by looking at the Laplace 
 asymptotic behavior of the partition function $Z_{N,a,\gb}$ as $N \to \infty$. More precisely,
 one may define the free energy $F(\gb)$ by
 \begin{equation}\label{DFR}
   F(\gb) := \lim_{N \to \infty} \frac{1}{N} \log\left(Z_{N,a,\gb} \right).
 \end{equation}

 The basic observation is that the free 
 energy is non-negative.  In fact, one has:
  \begin{equation}
   \begin{split}
  Z_{N,a,\gb} & \geq  \bE\left[ \exp\left(\gb \sum_{k=1}^N \ind_{ S_k \in [0,a] } \right) \ind_{ S_k > a, k=1,2, \ldots ,N}\right] \\
       & \geq \bP\left[ S_j > a, j = 1 \ldots, N\right]. \\
   \end{split}
  \end{equation} 
    Choose some $M >  a$ such that $\bP[ S_1 \in [a,M]] > 0$. Integrating over $S_{1}$, one gets:
 \begin{equation}
  \bP[ S_j > a, j = 1, \ldots, N] \geq \int_{[a,M]} h(t) \bP_t \left[ S_1 > a, \ldots, S_{N-1} > a \right] dt .
 \end{equation}  
  A consequence of fluctuation theory estimates for killed random walk (see \cite{JS2}[Lemma 3.1] for the discrete case)
  is that
  for fixed $M$, the quantity $ N^{3/2}\bP_t \left[ S_1 > a, \ldots, S_{N-1} > a \right]  \in [c,c']$ for every $N$ and every
 $t \in [a,M]$ where $c,c'$ are positive constants. Thus:
  \begin{equation}
   Z_{N,a,\gb} \geq  \frac{c }{N^{3/2}}\int_{[a,M]} h(t) dt. 
  \end{equation} 
  Therefore $F(\gb) \geq 0$ for every $\gb$. Since the lower bound has been obtained by ignoring the contribution of the paths 
 that touch the strip, one is led to the following:

 \begin{defi}
 The model $\{ \bP_{N,a,\gb} \}$ is said to be localized if $F(\gb) > 0$. 
 \end{defi}

   The relevance of this classical definition is discussed for example in \cite{GB}[Chapter 1] for closely related models. 
 
 It is easy to show that $F(\cdot)$ is a convex function, in particular it is a continuous since $ F(\gb) \leq \gb \vee 0$ for every 
 $\gb \in \bbR$. 
 On the other hand, $F(\cdot)$ is increasing since for fixed $N$, $Z_{N,a,\gb}$ is increasing in $\gb$. Therefore, there exists 
 a critical value $\gb_c^a \in \bbR$ such that the strip  wetting model is localized for $\gb > \gb_c^a$.

  \subsubsection{ Scaling limits of the system. } 
 
    A common feature shared by the strip wetting model and the classical homogeneous one is the fact that the measure $ \bP_{N,a,\gb}$ exhibits
 a remarkable decoupling between the contact level set $\mathcal{I}_N := \{ i \leq N, S_i \in [0,a] \}$ and the excursions of $S$ between 
 two consecutive contact points (see \cite{DGZ} for more details in the standard homogeneous pinning case). Conditioning on 
  $I_N = \{ t_1, \ldots, t_k\}$ and on $(S_{t_1}, \ldots, S_{t_k})$, the \textit{bulk} excursions 
 $e_i = \{ e_i(n) \}_n := \left\{ \{S_{t_i+n} \}_{ 0 \leq n \leq t_{i+1} - t_i} \right\}$
 are independent under  $ \bP_{N,a,\gb}$ and are distributed like the walk $(S',\bP_{ S_{t_i}})$ conditioned on the 
 event  $\left\{ S'_{t_{i+1} - t_i}  = S_{t_{i+1}} , S'_{t_i + j} > a, j \in \{ 1, \ldots, t_{i+1} - t_i -1 \} \right\}$. 
It is therefore clear that 
 to extract scaling limits on $ \bP_{N,a,\gb}$, one has to combine good control over the law of the contact set $\mathcal{I}_N$ 
 and suitable asymptotics properties of the excursions. 
 
  This decoupling is the basis to the resolution of the large scale limits of the laws $\{ \bP_{N,a,\gb} \}$; in fact, we can 
  show that 
 for $\gb > \gb_c^a$, the diffusive rescaling limit of the laws $\{ \bP_{N,a,\gb} \}$ is simply the null function, whereas 
 in the delocalized phase the limiting law is the brownian motion conditioned to stay positive, that is the 
 brownian meander. We stress that the proofs are quite similar although technically more involved that in the classical 
 homogeneous wetting model, see \cite{JS} for more details. 
 
 What is left by these considerations is the critical case. For simplicity, in this regime we describe the scaling limits 
 in terms of the limit of the sequence $\{ \tau_{(N)} \}_{N}$. We point out that  
  additionnal tightness conditions should be made on the free process  $S$ in order 
 to ensure the convergence in law of the entire trajectories towards the reflected brownian motion (see \cite{CaG} to get some 
 insight into this issue). 

 It turns out that the scaling limits of the set of contact times with the strip is almost the $1/2$-stable
 regenerative set; more precisely, we compute its density with respect to the law of $\cA_{1/2}$.
 
 \begin{theorem}\label{main2}
     Consider the set $\tau := \{ j \in \N, S_{j} \in [0,a] \}$ 
 where $S_{j}$ is distributed according to the law $\bP_{N,a,\gb_c^{a}}$. One has the convergence in law: 
  \begin{equation}
   \tau_{(N)} \Rightarrow \cB_{1/2} \cap [0,1]
  \end{equation}  
 where $\cB_{1/2}$ is a random closed set whose law is absolutely continuous with respect to the one of $\cA_{1/2}$ with 
 Radon Nykodym derivative given by 
\begin{equation}
   \frac{\pi}{2}\left(1-\max(\cA_{1/2}\cap[0,1])\right)^{1/2}.
\end{equation} 
 \end{theorem}
 
 We stress that the density term that appears in the above result is pushing the rightmost point of the process away from $1$ and 
 closer to the origin.

 \subsection{ Outline of the paper.}
 
 The remaining of this paper is divided into three main parts, the first one dealing with the generic Markovian renewal setup, 
 the second one being devoted to the application of the main result of the first part to the strip wetting model and the last one 
    focusing on results about random walks. More specifically, 
  the exposition of the paper will be organized as follows: 
  \begin{enumerate}
   \item in  section \ref{sec2}, we explicit the asymptotic behavior of the Markov mass renewal function which forms the cornerstone 
 to the proof of our two main results. We then show Theorem \ref{main1} by using this asymptotic behavior.
   \item in section \ref{sec4}, we describe the underlying Markov renewal structure inherent to the strip wetting model. 
   \item in section \ref{FTES}, we  describe some results issued from fluctuation theory for random walks which will be the basis of our 
 approach for the proof of Theorem \ref{main2}.
   \item in section \ref{sec6}, we give a suitable representation for the free energy of the strip wetting model and apply it 
 to the proof of Theorem \ref{main2}. 
   \item in section \ref{sec7}, we prove the results given in section \ref{FTES}, relying on 
 some recent results by Doney \cite{Do3}.
  \end{enumerate}

 \section{ Asymptotic equivalence of the mass Markov renewal function.}\label{sec2}

 \subsection{Notations.}

   Given two 
 kernels $A_{x,dy}(n), B_{x,dy}(n)$ we define
 their convolution 
 \begin{equation}
   (A \ast B)_{x,dy}(n) :=  \sum_{m=0}^{n} \int_{\cE} A_{x,dz}(m) B_{z,dy} (n-m)
 \end{equation}
  and the $k$-fold convolution of the kernel $A$ with itself will be denoted by $A^{\ast k}_{x,dy}$ where by definition 
  $A^{\ast 0}_{x,dy} := \gd_x(dy) \ind_{ n = 0}$.

 For $(x,y) \in \cE^{2}, \gl > 0$,
 we consider the Laplace transform measure $U_{x,dy}(\gl)$ associated to $(u(n,x,dy))_{n}$ defined by
 \begin{equation}
   U_{x,dy}(\gl) := \sum_{j \geq 0} e^{-\gl j} u(j,x,dy). 
 \end{equation}  
 
  Our main technical step will be to prove the following proposition: 
  \begin{prop}\label{mainprop}
    For every $x \in \cE$, as $\gl \searrow 0$, one has the following weak convergence: 
   \begin{equation}\label{impor}
 \frac{\Gamma(1-\ga) \gl^{\ga} L(1/\gl)  \bE_{\Pi^{(2)}}[\Phi/k] }{ \ga }   U_{x,dy}(\gl) \to  \Pi[dy].
   \end{equation}
 
  \end{prop}

    To prove Proposition \ref{mainprop}, we will need some technical results which we summarize in the next section.

  \subsection{ Proof of Proposition \ref{mainprop}}

   The Markov renewal equation writes: 
  \begin{equation}
   u(n,x,dy) = \sum_{k \geq 0} K_{x,dy}^{\ast k}(n) 
  \end{equation}  
  which implies 
   \begin{equation}
    U(n,x,dy) = \sum_{k \geq 0} \sum_{j = 0}^{n}  K_{x,dy}^{\ast k}(j).
   \end{equation} 

 For conciseness, we will denote by $z := (x,y)$  an element of $\cE^{2}$.  Note that Assumption \ref{as1} implies in
particular that $0 < \min_{\cE^{2}} \Phi \leq \max_{\cE^{2}} \Phi < \infty.$

For $z \in \cE^{2}, \gl > 0$, define
\begin{equation}
 \phi_{z}(\gl) :=  \bE_{x}[ e^{-\gl \tau_1}| J_1 = y]. 
\end{equation}  

 We are able to control the behavior of $\phi_{z}(\cdot)$ close to the origin, and this uniformly on $\cE^{2}$. 
 
  \begin{prop}\label{lem2}
    As $\gl \searrow 0$, one has:
    \begin{equation}
     \sup_{z := (x,y) \in \cE^{2}} \left| k(x,y) \frac{(1-\phi_{z}(\gl))}{\gl^{\ga} L(1/\gl)} - \frac{ \Phi(z) \Gamma(1-\ga)}{\ga  } \right| \to 0.
    \end{equation}  
  \end{prop}

 To prove Proposition \ref{lem2}, we will need the following result:
 
  \begin{lem}\label{lem8}
      Given $\rho > 0$ and a family of atomfree measures $V^{(n)}_{z}$ on $\bbR^{+}$ 
 indexed by $z \in \cE^{2}, n \in \N$. Denote by $\Psi^{(n)}_{z}(\gl)$ their Laplace transforms.
  Assume that they verify that for every $u > 0$, one has 
 \begin{equation}\label{Hyp}
    \lim_{n \to \infty} \sup_{z \in \cE^{2}} \left| V^{(n)}_{z}([0,u]) -u^{\rho} \right| = 0.
 \end{equation} 
   and  assume moreover that there exists $d > 0$  such that 
  $\sup_{(n,z)\in \N \times \cE^{2}} \Psi^{(n)}_{z}(d) < \infty$. Then 
   for every $\gl > d $:
  \begin{equation}
   \lim_{n \to \infty}   \sup_{z \in \cE^{2}}  \left| \Psi^{(n)}_{z}(\gl)  - \frac{\Gamma(\rho +1)}{\gl^{\rho}} \right| = 0.
  \end{equation} 
    \end{lem}

 Note that in our case, the sequence of measures $V^{(n)}$ will have a density with respect to Lebesgue measure,
 and thus will be without atoms.
 
 \begin{proof}[Proof of Lemma \ref{lem8}]

 Let  $t > 0$ be fixed; for any $\gl > a$, 
uniformly on $z \in \cE^{2}$, as $n \to \infty$, we have:
  \begin{equation}\label{Pre}
   \int_{[0,t]} e^{- \gl u} V^{(n)}_{z}(du) \to \int_{[0,t]} e^{- \gl u} \rho u^{\rho -1}du. 
  \end{equation} 
 Indeed, integrating by part in the left hand side above and making use of the fact that the sequence $V^{(n)}$ is atomfree, one obtains:
  \begin{equation}
      \int_{[0,t]} e^{- \gl u} V^{(n)}_{z}(du) 
   =   e^{-\gl t} V^{(n)}_{z}(t)    +   \gl \int_{[0,t]} e^{- \gl u} V^{(n)}_{z}(u) du  
  \end{equation} 
 and making use of the uniformity statement of equation \eqref{Hyp} (which allows us in particular to
 use dominated convergence) and reintegrating 
 by part,
  equation \eqref{Pre} is proved. 

 On the other hand, still for $\gl > d$, it is clear that 
   \begin{equation}
    \int_{[t,\infty)} e^{-\gl u } V^{(n)}_{z}(du) \leq e^{-(\gl -d) t} \sup_{(n,z)\in \N \times \cE^{2}} \Psi^{(n)}_{z}(d) 
   \end{equation} 
 which can be made arbitrarily small (independently of $z,n$) as soon as $t$ is large enough. Noting that 
  \begin{equation}
   \int_{[0,\infty)} e^{- \gl u} \rho u^{\rho -1}du = \frac{\Gamma(\rho +1)}{ \gl^{\rho}},
  \end{equation} 
 we get the result.

 \end{proof}
 
  Now we prove Proposition \ref{lem2}.
 
 \begin{proof}[Proof of Proposition \ref{lem2}]
 For $z \in \cE^{2}$, we define the (infinite mass) measure on $\bbR^{+}$
 $U_{z}([0,n]) := \sum_{j \leq n}  \bP_{x}[ \tau_1 > j, J_1 = y]$. The following equivalence  holds \textit{uniformly} on $\cE^{2}$:
 \begin{equation}\label{eqeqU}
  U_{z}([0,n]) \sim  \sum_{j \leq n}\frac{L(j)  \Phi(z) }{ \ga j^{\ga}} \sim \frac{n^{1-\ga} L(n) \Phi(z) }{\ga(1-\ga)}.
 \end{equation} 
   Indeed, we use the general fact that if $u_n(z)$ is a positive sequence depending on $z \in \cE^{2}$ verifying  
 $\sup_{z \in S}|u_n(z)| < \infty$ and, as $n \to \infty$,
  \begin{equation}
  \sup_{z \in \cE^{2}} \left| \frac{n^{\ga} u_{n}(z)}{L(n)} - \Phi(z) \right| \to 0,
 \end{equation} 
  then 
  \begin{equation}
   \sup_{z \in \cE^{2}} \left| \frac{1-\ga}{n^{1-\ga}L(n)}\sum_{j \leq n} u_{j}(z) - \Phi(z) \right| \to 0.
  \end{equation} 
  Recalling the standard equivalence
  \begin{equation}
   \sum_{j=1}^{n} \frac{L(j)}{j^{\ga}} \sim \frac{n^{1-\ga}}{1-\ga} L(n),
  \end{equation} 
  the proof of this convergence is straightforward. 
  
%

 We denote by $\Psi_{z}(\gl)$ the Laplace transform associated to the measure $U_{z}(\cdot)$. Integrating by part, we
 then have the equality which is valid for $z \in \cE^{2}, \gl \geq 0$:
 \begin{equation}\label{eqeqPhi}
\begin{split}
  1- \phi_{z}(\gl) & = (1-e^{-\gl})\sum_{n \geq 0} e^{-\gl n} \bP_{x}[\tau_1 > n|J_1 = y]  \\
   & = (1-e^{-\gl}) \frac{\Psi_{z}(\gl)}{ k(x,y)}.              
                 \end{split}
 \end{equation}

 Making use of equation \eqref{eqeqU}, we give the asymptotic behavior of $\Psi_{z}(\cdot)$ close to the origin; more precisely,
 uniformly on $z \in \cE^{2}$, the convergence 
  \begin{equation}\label{AMF}
   \Psi_{z}(\gl) \sim \frac{L(1/\gl)\Phi(z) \Gamma(1-\ga)}{\ga \gl^{1-\ga}}
  \end{equation} 
  holds as $\gl \searrow 0$.  Clearly, if this is true, making use of equation \eqref{eqeqPhi},
 Proposition \ref{lem2} will be proved.
 
  We define a sequence of measures $V^{(n)}_{z}$ on $\bbR^{+}$ by $V^{(n)}_{z}([0,u]) := \frac{U_{z}([0,nu])}{U_{z}([0,n])}$. 
 Equation \eqref{eqeqU} implies that the hypothesis of Lemma \ref{lem8} is verified by the sequence $V_{z}^{(n)}$ with 
 $\rho = 1- \ga$ with an arbitrary $d > 0$. It is plain that, for every fixed $\gl > 0$,
 \begin{equation}
  \Psi^{(n)}_{z}(\gl) = \frac{\Psi_{z}(\gl/n)}{U_{z}([0,n])}.
 \end{equation} 
  Making use of Lemma \ref{lem8}, as $n \to \infty$, the convergence 
 \begin{equation}
  \frac{\Psi_{z}(\gl/n)}{U_{z}([0,n])} \to \frac{\Gamma(2 - \ga)}{\gl^{1-\ga}}
 \end{equation} 
    holds uniformly on $\cE^{2}$. Hence we get:
  \begin{equation}
  \Psi_{z}(\gl/n) \sim \frac{\Phi(z)L(n) \Gamma(2-\ga)}{\ga(1-\ga) (\gl/n)^{1-\ga}}
  \end{equation} 
 and this still holds uniformly on $\cE^{2}$. Extending this convergence to any sequence which decreases to zero (which is 
 straightforward), making use of standard
 properties of slowly varying functions and of the Gamma function, one gets the equivalence \eqref{AMF}.
  \end{proof}

 To finally prove Proposition \ref{mainprop}, we will need the following easy lemma. 
 \begin{lem}\label{lem1} Let $(c_n)_{n}$ be a positive sequence  which converges towards
 $c > 0$, $u_k$ a positive sequence 
 which converges towards $u > 0$ and $g(\cdot)$ a function $\bbR^{+} \to \bbR$ which is $o(\gl)$ as $\gl \searrow 0$. 
 Then, as $\gl \searrow 0$, one has:  
    \begin{equation}\label{eq12}
      \sum_{k \geq 0} c_k e^{-k (\gl u_k + g(\gl))} \sim \frac{c}{u \gl}.
    \end{equation} 
   \end{lem}
 \begin{proof}[Proof of Lemma \ref{lem1}]
    Note that one can ignore the first $k_1$ terms in the sum appearing in equation \eqref{eq12},
  where $k_1 \geq 0$ is an arbitrarily large fixed integer. Let $\gep \in (0,\min(u,c))$ be fixed. 
 For $k_1$ large enough and $\gl$ small enough (independently from $k_1$), one has :
 \begin{equation}
     (c - \gep) \sum_{k \geq k_1} e^{-k \left(\gl (u+\gep)\right) } \leq  \sum_{k \geq k_1}
 c_k e^{-k \left(\gl u_k+g(\gl)\right)} \leq (c + \gep) \sum_{k \geq k_1} e^{-k \left(\gl (u-\gep) \right)}.
 \end{equation} 
  Therefore 
 \begin{equation}
        \frac{c-\gep}{c}\frac{u}{u+ \gep} \leq \liminf_{\gl \searrow 0} \frac{\gl u}{c} \sum_{k \geq 0} c_k e^{-k (\gl u_k + g(\gl))}  \leq 
         \limsup_{\gl \searrow 0} \frac{\gl u}{c} \sum_{k \geq 0} c_k e^{-k (\gl u_k + g(\gl))} \leq \frac{c+\gep}{c}\frac{u}{u- \gep}
 \end{equation} 
    and considering $\gep$ arbitrarily close to zero yields the statement.

 \end{proof}

  Consider now a function $f$ which is continuous and positive on $\cE$, and let $x \in \cE$ be fixed. 
 Making use of the Markov renewal equation and  of the fact that conditionally on the $J_i$'s, the 
 $(\tau_i-\tau_{i-1})_{i \geq 1}$ are independent, we get: 
  \begin{equation}\label{eqrell}
   \int_{\cE} f(y) U_{x,dy}(\gl) = \sum_{n \geq 0}  \int_{\cE} f(y) \int_{\cE^{n-1}}
 \prod_{j=0}^{n-1} \phi_{x_{j},x_{j+1}}(\gl)\bP_{x_{j}}[J_1 = x_{j+1}]
  \end{equation} 
  where $x_{0} := x$ and $x_{n} := y$. For $(u,v) \in \cE^{2}$, we define 
 \begin{equation}
  g_{u,v}(\gl) := 1-\phi_{u,v}(\gl) - \frac{\Gamma(1-\ga) \gl^{\ga}L(1/\gl) \Phi(u,v)}{\ga k(u,v)}.
 \end{equation} 

  Recall that  Proposition \ref{lem2} asserts that $\sup_{(u,v) \in \cE^{2}} |g_{u,v}(\gl)| = o(\gl^{\ga}L(1/\gl) )$ as $\gl \searrow 0$.

For fixed $n$ and fixed $z \in \cE^{2}$, one has:
  \begin{equation}\label{truc2}
  \begin{split}
   &  \int_{\cE^{n-1}} \prod_{j=0}^{n-1} \phi_{x_{j},x_{j+1}}(\gl)  \bP_{x_{j}}[J_1 = x_{j+1}] \\
 &  = \int_{\cE^{n-1}} \exp\left( \sum_{j=0}^{n-1} \log\left( 1-(1-\phi_{x_{j},x_{j+1}}(\gl)) \right) \right)
  \bP_{x}[J_1 = x_{1},  \ldots, J_{n} = x_n ] \\
   \end{split}
 \end{equation} 
  Given $\gep \in (0,1/2)$, there exists $\gl_0$ such that as soon as $\gl \in (0,\gl_0)$, for every $(u,v) \in \cE^{2}$,
  there exists $c_{u,v} (= c_{u,v}(\gl)) \in [1/2-\gep, 1/2+\gep]$  such that the last term above is equal to:
 \begin{equation}
   \begin{split}\label{truc56}
   & \int_{\cE^{n-1}} \exp\left(  \sum_{j=0}^{n-1} (1-\phi_{x_{j},x_{j+1}}(\gl))   + c_{x_j,x_{j+1}}(1-\phi_{x_{j},x_{j+1}}(\gl))^{2}
 \right) \bP[J_1 = x_{1}, \ldots, J_{n} = x_n ] \\
  & = \bE \left[ \exp \left(  \sum_{j=0}^{n-1} - \frac{\Gamma(1-\ga)\gl^{\ga} L(1/\gl)}{\ga }
 \frac{\Phi(J_j,J_{j+1})}{k(J_j,J_{j+1})}  - g_{J_j,J_{j+1}}(\gl)
 + c_{J_j,J_{j+1}}(1-\phi_{J_{j},J_{j+1}}(\gl))^{2} \right)   \ind_{J_n = y} \right]. 
  \end{split} 
 \end{equation} 

  Making use of the ergodic theorem(which holds in our case, since $J$ is 
 a regular Markov chain (see \cite{Fel2}[VIII,7 Theorem 1]), we get that 
  \begin{equation}
    \frac{1}{n} \sum_{j=0}^{n-1} \frac{\Phi(J_j,J_{j+1})}{k(J_j,J_{j+1})} \to
 \bE_{\Pi^{(2)}}\left[\frac{\Phi}{k}\right] = \int_{\cE} \Pi(dx) \Phi(x,y) \mu(dy)
  \end{equation} 
 $\bP_x$ almost surely. 
   When $n$ becomes large, the generic term  appearing in the sum of the right hand side of \eqref{eqrell} finally becomes equivalent to 
 \begin{equation}
  \begin{split}
    & \int_{y \in \cE} f(y) \bP_x \left[ J_n = y \right] \\
   &  \times  \exp \left( - n \left(  \frac{ \Gamma(1-\ga)\gl^{\ga} L(1/\gl)}{\ga }
     \bE_{\Pi^{(2)}} \left[\Phi/k \right] +  \bE_{\Pi^{(2)}} \left[g_{J_0,J_{1}}(\gl)
 + c_{J_0,J_{1}}(1-\phi_{J_{0},J_{1}}(\gl))^{2}\right] \right) \right)
    .  \end{split}
 \end{equation} 

   The function $ \gl \mapsto \bE_{\Pi^{(2)}} \left[g_{J_0,J_{1}}(\gl)
 + c_{J_0,J_{1}}(1-\phi_{J_{0},J_{1}}(\gl))^{2}\right]$ is then $o(\gl^{\ga}L(1/\gl))$. 
 
  On the other hand, making use of the classical ergodic theorem, as $n \to \infty$, the following convergence holds:
 \begin{equation}
  \int_{y \in \cE} f(y) \bP_x \left[ J_n = y \right]  \to \int_{\cE} f(y) \Pi[dy]. 
 \end{equation} 
 Finally,  we can apply Lemma \ref{lem1} to prove 
 Proposition \ref{mainprop}.

\subsection{Proof of Theorem \ref{mP2}.}

 To prove Theorem \ref{mP2}, we will use the following result, which is known as the \textit{extended continuity theorem} (see
 \cite{Fel2}[XIII Theorem 2a]):
 
\begin{theorem}
  Let $U_n$ be a sequence of measures defined on $\mathbb{R}^{+}$, and let $\phi_n$ denote their Laplace transforms. 
  For a fixed $d > 0$, if the sequence
 of functions $\phi_n$ converges pointwise towards a function $\phi$ on $(d,\infty)$, then the sequence of measures 
 $U_n$ converges weakly towards a measure $U$ with Laplace transform $\phi$.
\end{theorem}

 We consider a continuous function $f$ on $\cE$ and a fixed $x \in \cE$. Assume first that $f(\cdot)$  is positive; making 
use of Proposition  \ref{mainprop}, for $t \searrow 0$ and fixed $\gl > 0$, we get the convergence:
 
 \begin{equation}\label{Lap}
  \frac{\int_{\cE} U_{x,dy}(\gl t) f(y)}{\int_{\cE} U_{x,dy}(t ) f(y)} \to \frac{1}{\gl^{\ga}}.
 \end{equation} 
 
 Consider now the measure $m_{f}(\cdot)$ on $\mathbb{R}^{+}$ which is defined by
 
 \begin{equation}
  m_{f}([0,n]) :=  \int_{\cE} f(y) U(n,x,dy), 
 \end{equation} 
  and denote by $\Psi_{f}(\cdot)$ its Laplace transform. It is then quite clear that one may rewrite equation \eqref{Lap} as 
 
 \begin{equation}
  \frac{\Psi_{f}( \gl t)}{\Psi_{f}( t)} \to \frac{1}{\gl^{\ga}}
 \end{equation}   
 when $t \searrow 0$. On the other hand, $ \Psi_{f}( \gl t)$ is the Laplace transform associated to the measure $m_{f}(u/t)$, and thus 
 by making use of the extended continuity theorem, this implies the following convergence (for $t \searrow 0$): 
 \begin{equation}\label{FINLA}
  \frac{m_{f}(u/t)}{\Psi_{f}( t)} \to \frac{u^{\ga}}{\Gamma(1+\ga)}.
 \end{equation} 
   Finally, considering the convergence \eqref{FINLA} with $u=1$ along the subsequence $t_n = 1/n$, as $n \to \infty$, one gets:
  \begin{equation}\label{EQUIFI}
   \int_{\cE} f(y) U(n,x,dy) \sim \frac{\int_{\cE} U_{x,dy}(n) f(y)}{\Gamma(1+\ga)}
  \end{equation} 
   and making use of Proposition \ref{mainprop}, this is exactly the statement of Theorem \ref{mP2}. Making use of the linearity of the equivalence 
  \eqref{EQUIFI} and writing $f = f_{+} - (-f_{-})$ where $f_{+}$ (respectively $f_{-}$) is the positive (respectively the negative) 
 part of $f$, we are done.

 \subsection{ Proof of Theorem \ref{main1}. }\label{sec3}

 We finally prove Theorem \ref{main1}. 
 
  The compactness of $\cC_{\infty}$ implies that  every sequence of probability measures on 
 $\cC_{\infty}$ is tight, so that in order to check the convergence in law of a sequence $(\bP_n)$ towards $\bP$, one just 
 has to show finite dimensional convergence, namely that for every $t_1 < t_2 < \ldots < t_n$, the following weak convergence holds:
   \begin{equation}\label{TRUUU}
    \bP_N G_t^{-1} \Rightarrow  \bP_N G_t^{-1}
   \end{equation} 
  where $G_t: \cC_{\infty} \to [0,\infty]^{n}$ is defined by $G_t(F) := (d_{t_1}(F), \ldots,d_{t_n}(F))$.
 
  Let us  recall that for $t > 0$, the law of $d_t$ under $\bP_{\ga}$ is given by (see \cite{Be}[Proposition 3.1])
\begin{equation}
  \bP_{\ga} [d_t \in dy] = \frac{\sin(\ga \pi)}{\pi } \frac{t^{\ga}}{y(y-t)^{\ga}} \ind_{y > t} dy.
\end{equation}

    For conciseness, we show \eqref{TRUUU} in the $n=1$ case when $J$ starts from an initial state $x \in \cE$. Considering an arbitrary $y > t$, we have:
  \begin{equation}
  \begin{split}
      \bP_{N}[J_0 = x,  d_t > y] &  =  \sum_{k \geq 1} \int_{\cE} \bP_x \left[ \tau_{k} \leq Nt, J_k \in dv, \tau_{k+1} > Ny \right]  \\
     & =  \sum_{j \leq [Nt]} \int_{\cE} \left( \sum_{k \geq 1} \bP_x \left[ \tau_{k} = j, J_{k} \in dv \right] \right)
 \bP_{v} \left[ \tau_{1} >  Ny -j \right] \\ 
     & =   \sum_{j \leq [Nt]} \int_{\cE} u(j,x,dv)
 \bP_{v} \left[ \tau_{1} >  Ny-j\right]. \\   
  \end{split}
\end{equation} 
 
 Using Abel's summation, 
 we get: 
   \begin{equation}\label{finau}
\begin{split}
        & \bP_{N}[ J_0 = x, d_t > y]   =  \sum_{j \leq [Nt]} \int_{\cE} \left( U(j,x,dv) - U(j-1,x,dv) \right)
 \bP_{v} \left[ \tau_{1} >  Ny - j \right] = \\
       & \int_{\cE} U(Nt,x,dv)\bP_{v}[  \tau_1 > N (t-y) -1] \\
 &  - \sum_{1 \leq j \leq Nt} \int_{\cE} U(j,x,dv) \bP_{v}[ \tau_1 = Ny - j -1] - \bP_x[\tau_1 > Ny]    
                    \end{split}
   \end{equation}   
 
  Of course, $\bP_x[\tau_1 > Ny] \to 0$ as $N \to \infty$.
 
 As for the first term in the right hand side of equation \eqref{finau}, we get :
\begin{equation}
 \begin{split}
  &  \int_{\cE} U(Nt,x,dv)\bP_{v}[ \tau_1 > N (t-y) -1] \\
& \phantom{iiiiiiiii}  \to \int_{\cE} \frac{\ga t^{\ga}}{\Gamma(1+\ga)\Gamma(1-\ga)} \frac{\Pi(dv)}{\bE_{\Pi^{(2)}}[\Phi/k]}
 \frac{ \int_{\cE} \Phi(v,u) \mu(du)}{\ga (y-t)^{\ga}}  \\
  & \phantom{iiiiiiiii}  = \frac{\sin(\ga \pi)}{ \pi} \left(\frac{t}{y-t}\right)^{\ga}.
 \end{split}
\end{equation} 
 
 For the second term in the right hand side of equation \eqref{finau}, first note that 
 for $\gep \in (0,y)$, making use of the uniform convergence property for slowly varying functions, one has :
  \begin{equation}
  \begin{split}
  &    \sum_{\gep N \leq j \leq Nt} \int_{\cE} U(j,x,dv) \bP_{v}[ \tau_1 = Ny - j -1] \\
 & \phantom{iiiii} \sim \int_{\cE} \sum_{ \gep N \leq j \leq Nt} \frac{\ga j^{\ga} \Pi[dv]}{L(j) \Gamma(1+\ga)
 \Gamma(1-\ga) \bE_{\Pi^{(2)}}[\Phi/k]} 
 \times \frac{L(Ny-j-1) \int_{\cE} \Phi(v,u) \mu(du)}{(Ny-j-1)^{\ga+1} } \\
  & \phantom{iiiiiiiiiiiiiii} \sim \frac{\ga}{\Gamma(1-\ga) \Gamma(1+\ga)} \int_{\gep}^{t} \frac{u^{\ga}}{(y-u)^{\ga +1}} du. 
  \end{split} 
  \end{equation} 
 
 On the other hand, we get easily that 
 
 \begin{equation}
  \lim_{\gep \searrow 0} \lim_{N \to \infty}  \sum_{j \leq \gep N  } \int_{\cE} U(j,x,dv) \bP_{v}[ \tau_1 = Ny - j -1] = 0. 
 \end{equation} 
    
 Thus we are left with checking the equality:
 \begin{equation}
    \frac{\sin(\ga \pi)}{\pi} \int_{y}^{\infty}  \frac{t^{\ga}}{u(u-t)^{\ga}} du = 
\frac{\sin(\ga \pi)}{ \pi} \left(\frac{t}{y-t}\right)^{\ga} -
 \frac{\ga \sin(\ga \pi)}{ \pi} \int_{0}^{t} \frac{u^{\ga}}{(y-u)^{\ga +1}} du 
 \end{equation}  
   which is easy using integration by parts; just note that 
  
 \begin{equation}
    \int_{0}^{t} \frac{u^{\ga}}{(y-u)^{\ga +1}} du = 
\frac{1}{\ga} \left[ \left(\frac{u}{y-u}\right)^{\ga} \right]_{0}^{t} - \int_{0}^{t} \frac{u^{\ga -1}}{(y-u)^{\ga}} du
 \end{equation} 
 and finally the equality 
 \begin{equation}
   \int_{0}^{t} \frac{u^{\ga -1}}{(y-u)^{\ga}} du = \int_{y}^{\infty}  \frac{t^{\ga}}{u(u-t)^{\ga}} du
 \end{equation} 
  is obvious.

    \section{Application to the strip wetting model}\label{sec4}

 \subsection{ Asymptotics about the kernel of the free process.}\label{FTES}

  The following transition kernel will be of basic importance for the proof of Theorem \ref{main2}:  
  \begin{equation} 
 \begin{split}
              & F_{x,dy}(n) := \bP_{x} [S_1 > a, S_2 > a, \ldots, S_{n-1} > a, S_n \in dy] \ind_{ x,y \in {[0,a]}^{2}} \hspace{.2 cm} \text{if} \hspace{.2 cm}  n \geq 2, \\
              &  F_{x,dy}(1) := h(y-x) \ind_{ x,y \in [0,a]} dy.
                   \end{split}
  \end{equation}
 For all  $n \in \N$ and $x \in [0,a]$, the kernel $F_{x,dy}(n) $ has a density $f_{x,y}(n)$
 with respect to the Lebesgue measure restricted to $[0,a]$.
  
 We define the sequence of entry times of $S$ into the strip by $\tau_0 := 0$
 and $\tau_{n+1} := \inf \{ k > \tau_{n}, S_{k} \in [0,a] \}$.
   We also consider the process $(J_n)_{n \geq 0}$ where  $J_n := S_{\tau_n}$; the
 process $\tau$ is then a  Markov
 renewal process whose modulating chain is the Markov chain $J$, and with associated kernel $F_{\cdot,\cdot}(\cdot)$.

  Note that the kernel $F$ is defective, in the sense that
 $\int_{[0,a]}\sum_{k \geq 1}F_{x,dy}(k) = \bP_{x}[S_1 \geq 0] \leq \bP[S_1 \geq -a] < 1$ for every $x \in [0,a]$. 
 In particular, for every $x \in [0,a]$, the quantity $\bP_{x}[\tau_1 = \infty] := 1 - \int_{[0,a]}\sum_{k \geq 1}F_{x,dy}(k)$ is
 strictly positive.

 We denote by $l_N$ the cardinality of $\{ k \leq N | S_k \in [0,a] \}$, and we define
  $\overline{F}_{x}(k) := \int_{[0,a]} \sum_{j > k} F_{x,dy}(j) $. 
With these notations, we write the joint law of $( l_N, (\tau_n)_{  n \leq l_N}, (J_n)_{n \leq l_N})$
 under $\bP_{N,a,\gb}$ under the following form: 
  \begin{equation}\label{HPP}
 \begin{split}
   &  \bP_{N,a,\gb} [ l_N = k,\tau_j = t_j, J_j \in dy_i, i=1 ,\ldots, k ] \\
& \phantom{x}  = \frac{e^{\gb k}}{Z_{N,a,\gb}}  F_{0,dy_1}(t_1)F_{y_1,dy_2}(t_2-t_1)
 \ldots F_{y_{k-1},dy_k}(t_k-t_{k-1}) \left( \overline{F}_{y_{k}}(N-t_{k}) + \bP_{y_{k}}[\tau_1 = \infty] \right)
 \end{split}
    \end{equation} 
 where $k \in \N,  0 < t_1 < \ldots < t_k \leq N$ and $ (y_i)_{i=1, \ldots,k} \in \bbR^k$.

  To describe the asymptotic behavior of the kernel $F$, we will need some results issued from fluctuation theory for random walks.
 Let us collect some basic facts about this topic.

For $n$ an integer, we denote by $T_n$ the {\sl $n$th ladder epoch}; that is $T_0 := 0$
 and, for $n \geq 1$, $T_n := \inf\{ k \geq T_{n-1}, S_k > S_{T_{n-1}} \}$. 
We also introduce the so-called {\sl ascending ladder heights} $ (H_n)_{n \geq 0}$,
 which, for $ k \geq 1$, are given by $H_k := S_{T_k}$. Note that the process $(T,H)$ is 
a bivariate renewal process on $(\bbR^+)^2$. 
  In a similar way, one may define the strict descending ladder variables process $(T^-,H^-)$ as 
  $(T^-_0,H^-_0) := (0,0)$ and 
   \begin{equation}
    T^-_n := \inf\{ k \geq T_{n-1}, S_k < S_{T_{n-1}} \} \hspace{.6 cm} \text{and} \hspace{.6 cm} H^-_k := -S_{t_k^-}. 
   \end{equation}
 We define the following function: 
   \begin{equation}\label{DefPh}
   \Phi_a(x,y) :=  \frac{\bP [H_1^- \geq a-y]\bP [H_1 \geq a-x]}{\gs \sqrt{2\pi} } \hspace{ .2 cm}
  \ind_{ x,y \in [0,a]}.
   \end{equation}
   
 In the appendix, we show the following theorem, which will be the cornerstone of our approach: 
  \begin{theorem} \label{Pr} The following equivalence holds uniformly on $(x,y) \in [0,a]^2$: 
    \begin{equation}\label{EQC}
    n^{3/2}f_{x,y}(n) \sim \Phi_a(x,y).
    \end{equation}
      \end{theorem}

 It is a well known fact
 that the continuity of $h(\cdot)$ implies the continuity of the distribution function of $H_1$ as one can deduce from 
 the identity $ \bP[H_1 \in I] = \sum_{k \geq 1} \bP[T_1 = k, S_k \in I]$ which is valid for every interval $I$. Moreover, 
 as we assumed $n_o = 1$ in equation \eqref{hypn}, $\Phi_{a}$ is continuous and strictly positive on the whole square 
 $[0,a]^{2}$, as one has the equivalence
  \begin{equation}
   \bP[ H_1 > a] > 0 \Longleftrightarrow \bP[ S_1 > a] > 0 
  \end{equation} 
 and of course a similar statement holds for the first descending ladder height process.

  \subsection{Defining the free energy through an infinite dimensional problem.}\label{sec6} In this section,
 we define the free energy in a
  way that allows us to make use of the underlying Markov renewal structure inherent to this model.

  For $\gl  \geq 0$, we introduce the following kernel:
\begin{equation}
 B_{x,dy}^{\gl} := \sum_{n=1}^{\infty} e^{-\gl n} F_{x,dy}(n)
\end{equation}
 and the associated integral operator 
 \begin{equation}
   (B^{\gl}h)(x) := \int_{[0,a]} B_{x,dy}^{\gl} h(y). 
 \end{equation} 
 We then have the 
 \begin{lem} \label{Comp}
  For $\gl \geq 0, B_{x,dy}^{\gl}$ is a compact operator on the Hilbert space $L^2([0,a])$. 
 \end{lem}
 \begin{proof}[Proof]
   Of course $B_{x,dy}^{\gl}$ has a density with respect to the Lebesgue measure restricted to $[0,a]$ that we denote by $b^{\gl}(x,y)$. To show Lemma \ref{Comp},
  it is sufficient to show that $B_{x,dy}^{\gl}$ is actually Hilbert-Schmidt, that is that  
  \begin{equation} \label{L1}
 \int {b^{\gl}(x,y)}^2 \ind_{x,y \in [0,a]} dx dy < \infty.
 \end{equation}
  
 Note that Theorem \ref{Pr} entails in particular that 
 there exists $c > 0$ such 
that, for all $n \in \N$ and $x,y \in [0,a], f_{x,y}(n) \leq  \frac{c}{n^{3/2}}$.

 Making use of this inequality, it is then straightforward to show \eqref{L1}.
\end{proof}
    
  Lemma \ref{Comp} enables us to introduce $\gd^a(\gl)$, the spectral radius of the operator $B^{\gl}$. 
It is easy to check that $\gd^a(\gl) \in (0,\infty)$ for $\gl \geq 0$; $\gd^a(\gl)$ is an isolated
  and simple eigenvalue of  $B_{x,dy}^{\gl}$ (see Theorem 1 in \cite{Ze}). The function $\gd^a(\cdot)$ is non-increasing,
 continuous on $[0,\infty)$ and analytic on $(0,\infty)$ because the operator $B_{x,dy}^{\gl}$ has these properties. The analyticity
 and the fact that $\gd^a(\cdot)$ is not constant (as $\gd^a(\gl) \to 0$ as $\gl \to \infty$) force $\gd^a(\cdot)$ 
to be strictly decreasing.
  
  We denote by $(\gd^a)^{-1}(\cdot)$ its
  inverse function, defined on $(0,\gd^a(0)]$. We now define $\gb_c^a$ and $F^a(\gb)$ by:
  \begin{equation} \label{DefFEE}
   \gb_c^a := -\log(\gd^a(0)), \hspace{.2 cm} F^a(\gb) := ( \gd^a)^{-1}(\exp(-\gb) ) \hspace{.2cm} \text{if} \hspace{.2cm} \gb \geq \gb_c^a  \hspace{.2cm} \text{and} \hspace{.2cm}  0 \hspace{.2cm} \text{otherwise.}
\end{equation}
 Note that this definition entails in particular the analyticity of $F(\cdot)$ on $\bbR \setminus \{ \gb_c^{a} \}$. We discuss below 
 the relevance of this definition in accordance with the definition given in equation \eqref{DFR}.

  The fact that $b^{F^a(\gb)}(x,y) > 0$ for every $(x,y) \in [0,a]$ implies the uniqueness (up to a multiplication 
 by a positive constant) and the positivity almost everywhere of 
the so-called right and left Perron Frobenius eigenfunctions of $B_{x,dy}^{F^a(\gb)}$ ;
 more precisely, Theorem 1 in \cite{Ze} implies that there exist two functions $v_{\gb}(\cdot)$ and $w_{\gb}(\cdot)$ in $L^2([0,a])$ 
such that $v_{\gb}(x)> 0$ and $w_{\gb}(x)> 0$ for almost every $x \in [0,a]$, and moreover:  
    \begin{align}
    & \int_{y \in [0,a]} B_{x,dy}^{F^a(\gb)} v_{\gb}(y) = \left( \frac{1}{e^{\gb}} \wedge  \frac{1}{e^{\gb_c^a}} \right) v_{\gb}(x) \\
    & \int_{x \in [0,a]} w_{\gb}(x) B_{x,dy}^{F^a(\gb)} dx = \left( \frac{1}{e^{\gb}} \wedge  \frac{1}{e^{\gb_c^a}} \right) w_{\gb}(y) dy. 
 \end{align}
  Spelling out these equalities, we get that 
 \begin{equation}
  v_{\gb}(x) =  \frac{1}{\frac{1}{e^{\gb}} \wedge  \frac{1}{e^{\gb_c^a}}}
 \sum_{ n \geq 0} e^{-F^a(\gep)n} \int_{y \in [0,a]} f_{x,y}(n) v_{\gb}(y) dy,
 \end{equation} 
    which implies in particular the fact that $ v_{\gb}(\cdot)$ is positive and continuous on the
 whole $[0,a]$ and not only almost everywhere. Similarly,
 the function $ w_{\gb}(\cdot)$ is everywhere positive and continuous. 
  These considerations lead us to define the new kernel 
 \begin{equation}
  K_{x,dy}^{\gb}(n) := e^{\gb} F_{x,dy}(n) e^{-F^a(\gb)n}  \frac{v_{\gb}(y)}{v_{\gb}(x)}. 
\end{equation}

     It is then straightforward to check that 
  \begin{equation} \label{INVMP}
  \begin{split}
   & \int_{y \in \bbR} \sum_{n \in \N} K_{x,dy}^{\gb}(n) = \frac{e^{\gb}}{v_{\gb}(x)} \int_{y \in \bbR} \sum_{n \in \N} F_{x,dy}(n) e^{-F^a(\gb)n}v_{\gb}(y) \ind_{y \in [0,a]}  \\ 
   &  = \frac{e^{\gb}}{v_{\gb}(x)} \int_{y \in \bbR} B_{x,dy}^{F^a(\gb)} v_{\gb}(y) = 1 \wedge \frac{e^{\gb}}{e^{\gb_c^a}}. 
   \end{split}
   \end{equation}

   In particular, when $\gb = \gb_c^{a}$, $K_{\cdot,\cdot}^{\gb^{a}_{c}}$ is a non defective Markov renewal kernel, which satisfies 
 Assumption \ref{as1}. Indeed, when $n  \to \infty$, it is a consequence of Theorem \ref{Pr} that the following convergence holds uniformly 
 on $[0,a]^{2}$:
   \begin{equation}
    n^{3/2} \frac{K_{x,dy}^{\gb^{a}_{c}}(n)}{dy} \sim e^{\gb_{c}^{a}} \frac{ v_{\gb}(y)}{v_{\gb}(x)} \Phi_a(x,y),
   \end{equation} 
  which is equation \eqref{BASEQ} with $\ga = 1/2$ and $L(\cdot)$ a trivial slowly varying function. Let us define 
 $\tilde{\Phi}_{a}(x,y) := e^{\gb_{c}^{a}} \frac{ v_{\gb}(y)}{v_{\gb}(x)} \Phi_a(x,y)$. Note also that the strict positivity 
 of  $\tilde{\Phi}_{a}$ on the square $[0,a]^{2}$ implies that the condition about the strict positivity of the
 transition kernel of $J$ is also satisfied. 
  
  It is not hard to see that for every $\gb \in \bbR$,
  the measure $\Pi_{\gb}(dx) := v_{\gb}(x) w_{\gb}(x) dx$ is invariant for the Markov chain $J$; moreover, both $v_{\gb}(\cdot)$
  and $w_{\gb}(\cdot)$ being
 defined up to a multiplicative constant, one can use this degree of freedom to make $\Pi_{\gb}$ a probability measure on $[0,a]$. 
 
   Let us point out that it is not clear \textit{a priori} that the quantity we define in \eqref{DefFEE}
 actually coincides with the  definition given in equation \eqref{DFR}. We will not delve into this issue in  this paper; however, 
 we stress that this identification relies on applying the Markov renewal theorem to the kernel $K^{\gb}$ 
 in the localized phase (thus when the Markov renewal associated to $K^{\gb}$ is recurrent positive),
  and on direct (polynomial) asymptotics on the partition functions in the delocalized one (for more details,
 see the thesis \cite{JS}[Chapter 2]). For all 
 relevant purposes, we will use the definition given in equation \eqref{DefFEE} as a definition for the free energy. 

 \subsection{Proof of Theorem \ref{main2}.}

 We will need the asymptotic behavior of $Z_{N,a,\gb_c^{a}}(x)$, which is given in the next lemma:
 
  \begin{lem}\label{estZ} As $N \to \infty$ and for $x \in [0,a]$, the following equivalence holds:
  \begin{equation}
 \begin{split}
    Z_{N,a,\gb_c^{a}}(x) & \sim  N^{1/2} \frac{v_{\gb^{a}_c}(x) (1 - e^{-\gb_c^{a}})}
{ \pi e^{\gb_c^{a}}\int_{[0,a]^{2}} v_{\gb^{a}_c}(t) w_{\gb^{a}_c}(s) \Phi_a(s,t) ds dt}
 \int_{[0,a]} w_{\gb_c^{a}}(s) ds \\
   & \sim  N^{1/2} v_{\gb^{a}_c}(x) C_{a,\Phi}.
 \end{split}  
  \end{equation} 
  \end{lem}

 \begin{proof}[Proof of Lemma \ref{estZ}]

 For $(x,y,n) \in [0,a]^{2} \times \mathbb{N}$, we denote by $\gk$ the Markov renewal mass function
 associated to the kernel $K_{\cdot,\cdot}^{\gb^{a}_{c}}$, that is
 $\gk(n,x,dy) := \sum_{k \geq 1} (K^{\gb^{a}_{c}})_{x,dy}^{\ast k} (n)$.
 
  Summing over the last contact point of the process with the strip before time $N$, we have:
  \begin{equation}\label{bound}
 \begin{split}
       Z_{N,a,\gb_c^{a}}(x) & = v_{\gb^{a}_c}(x) \int_{[0,a]} \sum_{n=0}^{N} \gk(n,x,dy) \frac{1}{v_{\gb^{a}_c}(y)} \bP_{y}[ \tau_1 > N-n] \\  
 & +   v_{\gb^{a}_c}(x) \int_{[0,a]} \sum_{n=0}^{N} \gk(n,x,dy) \frac{1}{v_{\gb^{a}_c}(y)} \bP_{y}[ \tau_1 = \infty] \\
        & = v_{\gb^{a}_c}(x) \int_{[0,a]} \sum_{n=0}^{N} \gk(n,x,dy)   \sum_{j > N-n} \int_{[0,a]}
 \frac{K^{\gb_c^{a}}_{y,du}(j)}{v_{\gb_c^{a}}(u)} \\
    & +   v_{\gb^{a}_c}(x) \int_{[0,a]} \sum_{n=0}^{N} \gk(n,x,dy) \frac{1}{v_{\gb^{a}_c}(y)} \bP_{y}[ \tau_1 = \infty]. \\
                   \end{split}
  \end{equation}   
 
 Taking into account the fact that $v_{\gb_{c}^{a}}(\cdot)$ is continuous and positive on $[0,a]$, and that $K_{\cdot,\cdot}^{\gb^{a}_{c}}$
 is non defective (and thus that
 $\int_{[0,a]} \sum_{n=0}^{N} \gk(n,x,dy) \sum_{j > N-n} \int_{[0,a]} K^{\gb_c^{a}}_{y,du}(j) = \bP_{x}[S_1 > 0] \leq 1 $), 
 the first term in the right hand side of equation \eqref{bound} is bounded as $N$ becomes large. As for the second term, denoting by 
 $k_{\gb_c^{a}}(s,t) := \frac{1}{dt}\sum_{n \geq 1} K^{\gb_c^{a}}_{s,dt}(n)$, we make use 
 of Proposition \ref{mP2} in the case where $\cE = [0,a]$ is equipped with the  Lebesgue measure  to get the equivalence: 
  \begin{equation}
   \int_{[0,a]} \sum_{n=0}^{N} \gk(n,x,dy)  \frac{\bP_{y}[\tau_1 = \infty]}{ v_{\gb_c^{a}}(y)} \sim 
    \frac{N^{1/2} }{\pi} \int_{[0,a]} \frac{\Pi_{\gb_c^{a}}(dy)}{\bE_{\Pi^{(2)}}[\tilde{\Phi}_{\gb_c^{a}}/k_{\gb_c^{a}}]} \frac{\bP_{y}[\tau_1 = \infty]}{v_{\gb_c^{a}}(y)}.
  \end{equation} 
  Then we recall that $\Pi_{\gb_c^{a}}(du) = v_{\gb_c^{a}}(u) w_{\gb_c^{a}}(u) du $ and we make use of the equality
 \begin{equation}
  \int_{[0,a]} w_{\gb_c^{a}}(s)  \left( 1 - \sum_{j \geq 1} \int_{[0,a]} F_{s,dt}(j) \right) ds = 
\left(1 - e^{-\gb_c^{a}}\right) \int_{[0,a]} w_{\gb_c^{a}}(t) dt.
 \end{equation}
 Finally making use of equality \eqref{Inv} we get Lemma \ref{estZ}.
  
 \end{proof}

 We will also need the following deterministic lemma:

 \begin{lem}\label{DCG}
  Let $v_n$, $u_n$ and $w_n$ be positive sequences such that $v_n \sim w_n \sim \sqrt{n}$ and $u_n \sim n^{-3/2}$ as $n \to \infty$. For any 
 $(s,t) \in (0,1)^{2}$ such that $s < t$, the following convergence holds as $N \to \infty$:
 \begin{equation}\label{LFLF}
  \frac{1}{v_N} \sum_{i=1}^{sN} w_i \sum_{j=tN}^{N} (u_{j-i}-u_{j-i-1})v_{N-j} \to -\frac{3}{2} \int_{(u,v) \in [0,s] \times [t,1]} 
   \frac{\sqrt{u(1-v)}}{(v-u)^{5/2}} du dv.
 \end{equation}    
 \end{lem}

   \begin{proof}[Proof of Lemma \ref{DCG}]
  Let $\gep \in (0,1)$ and $l \geq 2$ be fixed.
 
  We write the left hand of the convergence appearing in equation \eqref{LFLF} as 
 \begin{equation}
  \frac{1}{v_N} \sum_{i=1}^{sN} w_i \sum_{m=0}^{l-1}V_{N,i,m}
 \end{equation} 
   where for $i \in [1,sN], m \in [0,l-1]$, we defined
 \begin{equation}
  V_{N,i,m} := \sum_{n=tN + (1-t)Nm/l}^{tN + (1-t)N(m+1)/l} (u_{n-i}-u_{n-i+1})v_{N-n}.
 \end{equation}
 Defining 
  \begin{equation} \begin{split}
         (\star) & := \sum_{n=tN + (1-t)Nm/l}^{tN + (1-t)N(m+1)/l} \left(u_{n-i}-u_{n-i+1}\right) \sqrt{N(1-t)\left(1-\frac{m}{l}\right)} \\
   &= \left( u_{N\left(t + \frac{(1-t)m}{l} -\frac{i}{N} \right) + N\left(\frac{1-t}{l}\right)}
 - u_{N\left(t + \frac{(1-t)m}{l} -\frac{i}{N} \right)}\right) \sqrt{N(1-t)\left(1-\frac{m}{l}\right)}.  
                   \end{split}
  \end{equation} 
 for  $N$ large enough, the following inequalities hold for every $i \in [1,sN], m \in [0,l-1]$:
 \begin{equation}
  (1-\gep) (\star) \leq V_{N,i,m} \leq 
  (\star) (1+\gep).
 \end{equation} 
 
  We get that, as $N \to \infty$, for every  $i \in [1,sN], m \in [0,l-1]$, one has: 
 \begin{equation} \begin{split}
          \left(t + \frac{(1-t)m}{l} -\frac{i}{N} \right)^{3/2} N^{3/2} 
  \left( u_{N\left(t + \frac{(1-t)m}{l} -\frac{i}{N} \right) + N\left(\frac{1-t}{l}\right)}
 - u_{N\left(t + \frac{(1-t)m}{l} -\frac{i}{N} \right)} \right) \\ 
 = \left(1 + \frac{\frac{1-t}{l}}{\left(t + \frac{(1-t)m}{l} -\frac{i}{N} \right)}\right)^{-3/2} - 1 + o(1),  
                  \end{split}
 \end{equation}  
 so that as $N \to \infty$, the following equivalence holds:
  \begin{equation}
   N \times (\star) \sim \left( \left(1 + \frac{\frac{1-t}{l}}{\left(t + \frac{(1-t)m}
{l} -\frac{i}{N} \right)}\right)^{-3/2} - 1 \right) \frac{\sqrt{(1-t)(1-m/l)}}{ \left(t + \frac{(1-t)m}{l} -\frac{i}{N} \right)^{3/2}} 
  \end{equation} 
   
 As a consequence, as $N \to \infty$, for every $l \geq 2$, we get the following inequalities:
  \begin{equation}
      \begin{split}
   (1- \gep) & \int_{u \in [0,s]} \sqrt{u} \sum_{m=0}^{l-1}
 \left( \left(1 + \frac{\frac{1-t}{l}}{\left(t + \frac{(1-t)m}{l} -u \right)}\right)^{-3/2} - 1 \right)
 \frac{\sqrt{(1-t)(1-m/l)}}{ \left(t + \frac{(1-t)m}{l} -u \right)^{3/2}} du  \\  
      & \leq \liminf_{N \to \infty} \frac{1}{v_N} \sum_{i=1}^{sN} w_i \sum_{m=0}^{l-1}V_{N,i,m} \\
   &  \leq \limsup_{N \to \infty} \frac{1}{v_N} \sum_{i=1}^{sN} w_i \sum_{m=0}^{l-1}V_{N,i,m} \\
    \leq (1 + \gep) & \int_{u \in [0,s]} \sqrt{u} \sum_{m=0}^{l-1}
 \left( \left(1 + \frac{\frac{1-t}{l}}{\left(t + \frac{(1-t)m}{l} -u \right)}\right)^{-3/2} - 1 \right) 
\frac{\sqrt{(1-t)(1-m/l)}}{ \left(t + \frac{(1-t)m}{l} -u \right)^{3/2}} du.
     \end{split}
  \end{equation} 

 For $l \to \infty$, making use of Riemann's sums and of the dominated convergence theorem, the integral appearing in the last term 
 of the above equation becomes equivalent to
 \begin{equation}
    -\frac{3}{2} \int_{u \in [0,s]} \int_{v \in [0,1]} \sqrt{u}  \frac{(1-t)^{3/2}
 \sqrt{1-v}}{\left( t + (1-t)v -u\right)^{5/2}} du dv,
 \end{equation} 
 from which we deduce the result. 
   \end{proof}

  Now we go to the proof of Theorem \ref{main2}.

 \begin{proof}[Proof of Theorem \ref{main2}]

As in the proof of Theorem \ref{main1}, we just have to show finite dimensional convergence.

 We denote by 
  $\bP_{\cB_{1/2}}$ the law of the limiting random set appearing in 
 Theorem \ref{main1} and by $\bP_{1/2}$ the law of 
 the regenerative set with index $1/2$. For $F \in \cF$, we also define  $g_t(F) = \sup ( F \cap [0,t))$.
 
  Let us first compute the limiting quantity we are looking for. For $0< s < t < 1$, we write:
\begin{equation}\label{AFFF}
 \begin{split}
   & \bP_{\cB_{1/2}} (d_s > t)  = \bP_{\cB_{1/2}} (g_t < s, g_1 = g_s) + \bP_{\cB_{1/2}} (g_t < s, g_1 > t) \\
  &  =  \bP_{\cB_{1/2}} ( g_1 \leq s)  \\
   & + \frac{\pi}{2} \int_{u \in [0,s]} \int_{v \in [t-u, 1-u]} \int_{w \in [u+v,1]}
 \bP_{1/2}(g_t \in du,d_t - g_t \in dv) \sqrt{1-w} \bP_{1/2}(g_1 \in dw|d_t = u+v). 
 \end{split}
\end{equation} 
 
 We first make use of the explicit law of $g_1$  (see 
  \cite{Be}[Proposition 3.1]), and we deduce the following equalities:
 \begin{equation}
  \begin{split}
        \bP_{\cB_{1/2}} (d_s > 1) & = \frac{\pi}{2} \int_{[0,s]} \sqrt{1-u} \bP_{1/2} ( g_1 \in du) \\
   & = \frac{\pi}{2} \int_{[0,s]} \sqrt{1-u} \frac{1}{\pi} \frac{1}{\sqrt{u(1-u)}} du \\
   & = \sqrt{s}.
   \end{split}                
   \end{equation}

 Then we note that for every $x \in (t,1)$ and $w \in (x,1)$, the following holds:
  \begin{equation}
   \bP_{1/2}(g_1 \in dw|d_t = x) = \bP_{1/2}(x + g_{1-x} \in dw) = \frac{1}{\pi} \frac{dw}{\sqrt{(w-x)(1-w)}}.
  \end{equation}  
 
  On the other hand, the following equality is well known (see \cite{Bo}):
  \begin{equation}
   \bP_{1/2}(g_t \in du,d_t - g_t \in dv) = \frac{du dv }{2 \pi \sqrt{uv^{3}}}.
  \end{equation}

 Finally one has:

  \begin{equation}\label{CLO}
   \bP_{\cB_{1/2}} (d_s > t)  = \sqrt{s} + \frac{1}{2 \pi} \int_{(u,v) \in [0,s]\times[t,1]}     \sqrt{\frac{1-v }{ u(v-u)^{3}} }
    du dv.
  \end{equation}

 In the same spirit as for the proof of Theorem  \ref{main1}, we then establish the convergence of $\bP_{N,a,\gb_{c}^{a}} (d_s^{N} > t)$
 towards $  \bP_{\cB_{1/2}} (d_s > t)$ as $N \to \infty$, thus proving Theorem \ref{main2}.  A computation similar to the one we made 
 for the proof of Lemma \ref{estZ} leads to the following convergence (for $s \in (0,1)$):
 \begin{equation}
  \bP_{N,a,\gb_{c}^{a}} (d_s^{N} > 1) \to \sqrt{s}. 
 \end{equation} 
 
 Thus we are left with computing the limit of the second term in the right hand side of the equality below:
 \begin{equation}
  \bP_{N,a,\gb_{c}^{a}} (d_s^{N} > t) = \bP_{N,a,\gb_{c}^{a}} (d_s^{N} \geq 1 ) + \bP_{N,a,\gb_{c}^{a}} (d_s^{N} \in (t,1)).
 \end{equation} 
 We write:
  \begin{equation}
 \begin{split}
  \bP_{N,a,\gb_{c}^{a}} (d_s^{N} \in (t,1)) = \frac{v_{\gb^{a}_c}(0)}{ Z_{N,a,\gb_c^{a}} } 
   \sum_{n=0}^{sN}\int_{x \in [0,a]}  \gk ( n,0,dx) \\
     \times \int_{y \in [0,a]}
 \sum_{m=tN}^{N} \frac{K_{x,dy}^{\gb_c^{a}}(m-n)}{v_{\gb_c^{a}}(y)}
 Z_{N-m,a,\gb_c^{a}}(y). 
 \end{split}
  \end{equation} 

  Once again, we do not have access to the local behavior of the quantities $\gk (n,0,x)$
 for $n$ large, 
 thus we use integration by part as in the proof of Theorem \ref{main1} and we have:
  
  \begin{equation}\label{FINI}
 \begin{split}
  & \bP_{N,a,\gb_{c}^{a}} (d_s^{N} \in (t,1)) = \frac{v_{\gb^{a}_c}(0)}{ Z_{N,a,\gb_c^{a}} } \int_{x \in [0,a]}
  \sum_{j=0}^{sN}\gk(j,0,dx) \sum_{m = tN}^{N} \int_{y \in [0,a]} \frac{K_{x,dy}^{\gb^{a}_c}(m-sN) }{v_{\gb^{a}_{c}}(y)}
 Z_{N-m,a,\gb_c^{a}}(y) \\
 & + \frac{v_{\gb^{a}_c}(0)}{ Z_{N,a,\gb_c^{a}} } \sum_{n=0}^{sN-1} \int_{(x,y) \in [0,a]^{2}} 
 \sum_{j=0}^{n}  \gk(j,0,dx) 
 \sum_{m=tN}^{N}  \frac{K_{x,dy}^{\gb^{a}_c}( m-n) - K_{x,dy}^{\gb^{a}_c}( m-n-1) }{v_{\gb^{a}_{c}}(y)}
 Z_{N-m,a,\gb_c^{a}}(y)\\
 &  \phantom{iiiiiiii} + \frac{v_{\gb^{a}_c}(0)}{ Z_{N,a,\gb_c^{a}} } \int_{(x,y) \in [0,a]^{2}}
 \sum_{m = tN}^{N} \frac{K_{0,dy}^{\gb^{a}_c}(m) }{v_{\gb^{a}_{c}}(y)}
 Z_{N-m,a,\gb_c^{a}}(y) \\
  & =: (\star) + (\star \star) + (\star \star \star). 
 \end{split}
  \end{equation} 

  Making use of Lemma \ref{estZ}, it is straightforward to check that $(\star \star \star)$ is
 $\mathcal{O}(N^{-1/2})$ as $N \to \infty$.
 As 
 $N \to \infty$, the following holds:
   
  \begin{equation}
   \begin{split}
     (\star) & \sim \frac{1}{ N^{1/2} C_{a,\Phi}} \int_{ x \in [0,a]} \frac{\sqrt{sN}}{\pi} \frac{  \Pi_{\gb_{c}^{a}}(dx)}
{ \bE_{\Pi^{(2)}}[\tilde{\Phi}_{\gb_c^{a}}/k_{\gb_c^{a}}]} \sum_{m = tN}^{N} \int_{y \in [0,a]}
  \sqrt{N-m} \frac{ C_{a,\Phi}  \tilde{\Phi}_{\gb_c^{a}}(x,y)}{(m-sN)^{3/2}}   dy\\
   & \to  \frac{\sqrt{s}}{\pi} \int_{t}^{1} \frac{\sqrt{1-u}}{(u-s)^{3/2}} du.
   \end{split}
  \end{equation} 
 
  Making use of Lemmas \ref{estZ},
 \ref{DCG} and of the asymptotic behavior of $K_{x,dy}(n)$(which holds uniformly on $[0,a]^{2}$), we easily get the equivalence:
 \begin{equation}
  (\star \star )  \sim -\frac{3}{2 \pi}  \int_{(u,v) \in [0,s] \times [t,1]} \frac{\sqrt{u(1-v)}}{(v-u)^{5/2}} du dv.                
 \end{equation} 
 
 Putting everything together, we finally have the following convergence:
  \begin{equation}\label{convfin}
   \bP_{N,a,\gb_{c}^{a}} (d_s^{N} \in (t,1)) \to \frac{\sqrt{s}}{\pi} \int_{t}^{1} \frac{\sqrt{1-u}}{(u-s)^{3/2}} du 
- \frac{3}{2 \pi}  \int_{(u,v) \in [0,s] \times [t,1]} \frac{\sqrt{u(1-v)}}{(v-u)^{5/2}} du dv.
  \end{equation} 

 On the other hand, integrating by part, we get the equalities:
  
 \begin{equation}
\begin{split}
     \int_{(u,v) \in [0,s]\times[t,1]}  \sqrt{\frac{1-v }{ u(v-u)^{3}} } & = \left[ \int_{v \in[t,1]} 2 \frac{\sqrt{u(1-v)}}{(v-u)^{3/2}} dv \right]_{u=0}^{s} 
 - \frac{3}{2}   \int_{(u,v) \in [0,s] \times [t,1]} \frac{2\sqrt{u(1-v)}}{(v-u)^{5/2}} du dv   \\
         & = 2 \sqrt{s} \int_{t}^{1} \frac{\sqrt{1-v}}{(v-s)^{3/2}} dv
 - \frac{3}{2}   \int_{(u,v) \in [0,s] \times [t,1]} \frac{2\sqrt{u(1-v)}}{(v-u)^{5/2}} du dv.
 \end{split}
 \end{equation} 

 With the help of the above, comparing the equality \eqref{CLO} and the convergence appearing in equation \eqref{convfin}, we are done.

  \end{proof}

   \section{Appendix}\label{sec7}

  The aim of this section is to prove Theorem \ref{Pr}.

 Let us consider the renewal function $U(\cdot)$
 associated to the ascending ladder heights process:
 \begin{equation}
 U(x) := \sum_{k=0}^{\infty} \bP [ H_k \leq x] = \bE[ \mathcal{N}_x] = \int_{0}^{x} \sum_{m=0}^{\infty} u(m,y) dy
\end{equation}
  where $\mathcal{N}_x$ is the cardinality of $\{ k \geq 0, H_k \leq x \}$ and
 $ u(m,y) := \frac{1}{dy} \bP[ \exists k \geq 0, T_{k} = m, H_{k} \in dy]$ is the renewal mass function associated to $(T,H)$.
 It follows in particular from this definition 
 that $U(\cdot)$ is a subadditive increasing function, and in our context it is also continuous. Note also that $U(0) = 1$. 
  We denote by $V(x)$
  the analogous quantity for the process $H^-$, and by $v(m,y)$ the renewal mass function associated
 to the descending renewal $(T^{-},H^{-)}$.

  For $u > 0$, we define $\tau_u := \inf\{ k \geq 1, S_k \geq u \}$. As we did for the Markov renewal kernel, for any
  $k \geq 1$ and for any $B \in \gs(S_1, \ldots, S_{k-1})$,  we introduce
 the useful notation:
   \begin{equation}
     \frac{1}{dx} \bP[ B, S_k \in dx] =: \bP[ B, S_k = x]
   \end{equation} 
  Our proof is based on the following result, which is a specialization of results due to Doney \cite{Do3}[Proposition 18]: 
   \begin{theorem}\label{Dooo}
 For $x,y \in \mathbb{R}^{+}$, let us define two sequences $x_n := x(\gs^{2}  n)^{-1/2} $
 and  $y_n := y(\gs^{2} n)^{-1/2} $.
 For any fixed $\Delta_0 > 0$, the following estimate holds uniformly for $\Delta \in (0,\Delta_0)$:
       Uniformly as $x_n \vee y_n \to 0$, 
   \begin{equation}\label{esn}
    \bP[ S_n \in (x-y-\Delta,x-y], \tau_x > n ] \sim \frac{U(x) \int_{y}^{y + \Delta}V(w)dw}{\gs\sqrt{2 \pi } n^{3/2}}.
   \end{equation} 
  Moreover, there exists
 a constant $c > 0$ such that the left hand side of equation \eqref{esn} 
 is dominated by a multiple of the right hand side. 
   \end{theorem}

  \begin{proof}[Proof of Theorem \ref{Pr}]

     We denote by $\tilde{S}$ the random walk whose transitions are given by
   \begin{equation}
    \bP_{x}[\tS_1 \in dy] := h(x-y) dy
   \end{equation} 
 and for positive $u$, by $\ttau_u$ the quantity $\inf \{ k \geq 1, \tS_k > u \}$. 
 
  Let  $\gep \in (0,1)$ be fixed.
  Integrating over the first entry point of $S$ into $(a,\infty)$ and on the last 
  location of $S$ in $(a,\infty)$ before hitting the strip, for $x,y \in [0,a]^{2}$, one gets the equalities:
  \begin{equation}\label{FTCHH} \begin{split}
                & f_{x,y}(n) = 
                 \bP_{x}[ S_1 > a, \ldots, S_{n-1} > a, S_n = y] \\
& \phantom{iiiiii} = \int_{(a,\infty)^{2}} h(u-x) \bP_{u}\left[ S_1 > a, \ldots, S_{n-3} > a, S_{n-2} = v\right] h(y-v) du dv \\
 & \phantom{iiiiii} =  \int_{(a,\infty)^{2}} h(u-x) \bP\left[ S_1 > a-u, \ldots, S_{n-3} > a-u, S_{n-2} = v-u\right] h(y-v) dudv \\
  & \phantom{iiiiii} =  \int_{(a,\infty)^{2}} h(u-x) \bP\left[ \ttau_{u-a} > n-2, \tS_{n-2} = u-v\right] h(y-v) du dv \\
    & \phantom{iiiiii} =  \int_{\cD_{1}^{n}} \ldots +
 \int_{\cD_{2}^{n}} \ldots
  =: \cI_{1}^{n} +  \cI_{2}^{n},
                 \end{split}
  \end{equation} 
  where we defined the subsets of $(a,\infty)^{2}$:
   \begin{equation}
    \begin{split}
     & \cD_{1}^{n} := \left\{ (u,v) \in (a,\infty)^{2}, u \vee v \leq \gep \sqrt{n} \right\} \\
     & \cD_{2}^{n} := \left\{ (u,v) \in (a,\infty)^{2}, u \vee v \geq \gep \sqrt{n}  \right\}. \\
    \end{split}
   \end{equation}

 We approximate the integral $\cI_{1}^{n}$ by below and by above by breaking the range of integration over $v$ 
 into subintervals of length
 $\eta > 0$, then we use the estimate given in equation \eqref{esn} and finally let $\eta \searrow 0$ to get 
 the convergence:  
   \begin{equation}\label{decc}
   \lim_{\gep \searrow 0} \lim_{n \to \infty} \sup_{x,y \in [0,a]^{2}}
  \left|\gs \sqrt{2 \pi}n^{3/2} \cI_{1}^{n}  - \int_{(a,\infty)} h(u-x) V(u-a) du \int_{(a,\infty)}  U(v-a) h(y-v)  dv \right| = 0.
   \end{equation}  
  Then we note that 
  \begin{equation} \begin{split}
         & \int_{(a,\infty)} h(u-x) V(u-a) du   = \int_{0}^{\infty} V(u) h(u+a-x)du \\
    & \phantom{iiiii} = \int_{0}^{\infty} h(u+a-x)du + \int_{0}^{\infty} \int_{ y \in [0,u]}\sum_{m \geq 1} v(m,w)  h(u+a-x) du dw\\
   & \phantom{iiiii} = \bP[ H_1 \geq a - x, T_1 = 1] + \int_{ y \in [0,\infty)} \left( \int_{u \in [y,\infty)} h(u+a-x) du \right) \sum_{m \geq 1} v(m,y) dy.  
                   \end{split}
  \end{equation} 
    We use of the well known \textit{duality lemma} (see \cite{Fel2}[chapter XII]),
 which asserts that the following equality holds:
   \begin{equation}
    v(m,y) =  \bP[ S_1 \leq 0, \ldots, S_m \leq 0, -S_m = y]. 
   \end{equation}  
    Making use of the Markov property, we then get the equalities:
   \begin{equation}
    \begin{split}
 & \int_{(a,\infty)} h(u-x) V(u-a) du   \\
 &  = \bP[ H_1 \geq a-x, T_1 = 1] + \sum_{ m \geq 1} \int_{0}^{\infty} \bP[ S_1 \leq 0, \ldots, S_{m} \leq 0, S_m = -y] \bP[S_1 \geq y+a-x] \\
 &  = \bP[ H_1 \geq a-x, T_1 = 1] + \sum_{ m \geq 1} \bP[ H_1 \geq a-x, T_1 = m]\\
 &  = \bP[ H_1 \geq a-x].
    \end{split}
   \end{equation} 

 A very similar computation yields the equality 
   \begin{equation}
   \int_{(a,\infty)}  U(v-a) h(y-v)  dv =  \bP[ H_1^{-} \geq a-y].
   \end{equation} 
 
  To complete the proof of Theorem \ref{Pr}, we are left with showing that
  \begin{equation}\label{TS}
   n^{3/2} \sup_{(x,y) \in [0,a]^{2}}  \cI_{2}^{n} \to 0
  \end{equation} 
  as $n \to \infty$.


 
 We recall that as the density $h(\cdot)$ is bounded, it is a consequence of 
 Stone's local limit theorem that there exists a constant $c > 0$ such that:
  \begin{equation}
    \sup_{n \in \N} \sup_{x \in \bbR} \sqrt{n} \bP[ S_n = x] \leq c. 
  \end{equation} 
 
 Relying on this domination, the convergence \eqref{TS} is easy. In fact, for fixed $\gep > 0$,
 as soon as $n$ is large enough ($c > 0$ is a constant which may vary from line to line and
 which is independent from $x$ and $y$):
   \begin{equation} \begin{split}
 \sup_{(x,y) \in [0,a]^{2} }n^{3/2}  \cI_{2}^{n} & \leq n^{3/2} \sup_{(x,y) \in [0,a]^{2} }
 \int_{ \cD_{2}^{n}} h(u-x) \bP\left[ \tS_{n-2} = u-v\right] h(y-v) du dv \\
  & \leq c  n \sup_{(x,y) \in [0,a]^{2} }
 \int_{ \cD_{2}^{n}} \left(\frac{u \vee v}{\gep \sqrt{n}}\right)^{2} h(u-x) h(y-v) du dv \\
  & \leq \frac{c}{\gep^{2}} \int_{\gep \sqrt{n}-a} u^{2} h(u) du \\
  & \leq \frac{c}{\gep^{2}} \int_{\gep \sqrt{n}/2} u^{2} h(u) du
  \end{split}
   \end{equation} 
 and recalling that $X \in L^{2}$, the last inequality is sufficient to conclude. 

 
 \end{proof} 
  
  \textbf{ Acknowledgement}: The author is very grateful to Francesco Caravenna for constant support during this work.

 \bibliographystyle{alpha}

  \bibliography{BibliCC}

\end{document}